\documentclass[referee,pdflatex,sn-mathphys-ay]{sn-jnl}

\usepackage{graphicx}%
\usepackage{multirow}%
\usepackage{amsmath,amssymb,amsfonts}%
\usepackage{amsthm}%
\usepackage{mathrsfs}%
\usepackage[title]{appendix}%
\usepackage{xcolor}%
\usepackage{textcomp}%
\usepackage{manyfoot}%
\usepackage{booktabs}%
\usepackage{algorithm}%
\usepackage{algorithmicx}%
\usepackage{algpseudocode}%
\usepackage{listings}%

\usepackage{dsfont}
\usepackage{latexsym}
\usepackage{amssymb}
\usepackage{amsmath}
\usepackage{amsthm}
\usepackage{paralist}
\usepackage{upgreek}
\usepackage{soul}
\usepackage{subcaption}
\usepackage{caption}
\usepackage{rotating}
\usepackage{tikz}
\usetikzlibrary{positioning}
\usepackage{blkarray}
\usepackage{mathtools}
\usepackage{pgfplots}
\usepackage{fancyhdr}
\usepackage{color}
\usepackage{tocbibind}
\usepackage{fancyhdr}
\usepackage{epsfig}
\usepackage{xcolor}
\usepackage{lineno,hyperref}
\hypersetup{colorlinks = true, linkcolor = blue, anchorcolor =red, citecolor = blue,filecolor = red, urlcolor = red, pdfauthor=author}

\usepackage{rotating}
\usepackage{url}
\usepackage{lscape}
\usepackage{float}
\usepackage{booktabs}

\DeclarePairedDelimiter\abs{\lvert}{\rvert}

\let\oldabs\abs
\def\abs{\@ifstar{\oldabs}{\oldabs*}}

\DeclareMathOperator*{\argmin}{arg\,min} 
\DeclareMathOperator*{\argmax}{arg\,max} 
\DeclareMathAlphabet\mathbfcal{OMS}{cmsy}{b}{n}

\theoremstyle{thmstyleone}%
\newtheorem{theorem}{Theorem}
\newtheorem{proposition}[theorem]{Proposition}%

\usepackage{lmodern}
\usepackage{fix-cm} 

\theoremstyle{thmstyletwo}%
\newtheorem{example}{Example}%
\newtheorem{remark}{Remark}%

\theoremstyle{thmstylethree}%

\newenvironment{examplecont}[1]
{\renewcommand\theexample{#1}\example}
{\endexample}

\pgfplotsset{compat=1.18}

\begin{document}

\title[On \texorpdfstring{$\lambda$}{}-Cent-Dians and Generalized-Center for Network Design]{On \texorpdfstring{$\lambda$}{}-Cent-Dians and Generalized-Center for Network Design: Definitions and Properties}


\author[1,2]{\fnm{V\'ictor} \sur{Bucarey}}\email{victor.bucarey@uoh.cl}
\equalcont{These authors contributed equally to this work.}

\author*[3]{\fnm{Natividad} \sur{Gonz\'alez-Blanco}}\email{ngonzalez@uloyola.es}

\author[4,5]{\fnm{Martine} \sur{Labb\'e}}\email{mlabbe@ulb.ac.be}
\equalcont{These authors contributed equally to this work.}

\author[6,7]{\fnm{Juan A.} \sur{Mesa}}\email{jmesa@us.es}
\equalcont{These authors contributed equally to this work.}

\affil[1]{\orgdiv{Institute of Engineering Sciences}, \orgname{Universidad de O'Higgins}, \orgaddress{\city{Rancagua}, \country{Chile}}}

\affil[2]{\orgdiv{Instituto Sistemas Complejos de Ingenier\'ia (ISCI)}, \orgaddress{\city{Santiago Centro}, \country{Chile}}}

\affil[3]{\orgname{Universidad Loyola Andalucía}, \orgdiv{Departamento de M\'etodos Cuantitativos}, \orgaddress{\city{Dos Hermanas}, \country{Spain}}}

\affil[4]{\orgdiv{D\'epartement d'Informatique}, \orgname{Universit\'e Libre de Bruxelles}, \orgaddress{\city{Brussels}, \country{Belgium}}}

\affil[5]{\orgdiv{Inria Lille-Nord Europe}, \orgaddress{ \city{Villeneuve d'Ascq}, \country{France}}}

\affil[6]{\orgdiv{Departamento de Matem\'atica Aplicada II}, \orgname{Universidad de Sevilla}, \orgaddress{\city{Sevilla}, \country{Spain}}}

\affil[7]{\orgdiv{Instituto de Matem\'aticas de la Universidad de Sevilla}, \orgname{Universidad de Sevilla}, \orgaddress{\city{Sevilla}, \country{Spain}}}


\abstract{In this paper, we extend the notions of $\lambda$-cent-dian and generalized-center from Facility Location Theory to the more intricate domain of Network Design. Our focus is on the task of designing a sub-network within a given underlying network while adhering to a budget constraint. This sub-network is intended to efficiently serve a collection of origin/destination pairs of demand. 
	
The $\lambda$-cent-dian problem studies the balance between efficiency and equity. We investigate the properties of the $\lambda$-cent-dian and generalized-center solution networks under the lens of equity, efficiency, and Pareto-optimality. We finally prove that the problems solved here are NP-hard.
}


\maketitle

\section{Introduction}\label{sec1}

\textit{Center} and \textit{median} problems in graphs and Euclidean spaces constituted the core of Location Science in the late 50s and 60s of the past century. Whereas median problems aim at maximizing the system's efficiency, center ones try to maximize the effectiveness or equity. In the median problem, the objective is to find one or more facility locations (points) in a given space, such that the normalized sum of weighted distances from the demand points to their closest facility is minimized. In contrast, the center problem seeks to find one or more facility locations to minimize the largest distance from a demand point to the nearest facility. 

On the one hand, the median problem is well-suited for situations where the primary objective is cost minimization or profit maximization within the system. On the other hand, the center problem is ideal for scenarios where the goal is to ensure that the farthest point is as close as possible to the facility, as in locating emergency facilities. However, when taken separately, these two objectives do not adequately address many real-world problems that require balancing efficiency and equity.

In \cite{halpern1976location}, the term $\lambda$-cent-dian was first introduced for location problems whose objective is to minimize a linear convex combination of center and median objectives, denoted by $F_c$ and $F_m$ respectively, i.e. the $\lambda$-cent-dian objective is $H_{\lambda} = \lambda F_c + (1-\lambda)F_m$. Subsequently, in \cite{halpern1978finding}, the author proved that the $\lambda$-cent-dian of a graph lies on a path connecting the center and the median, and provided a procedure to find the $\lambda$-cent-dians for all possible combinations given by the different values of $\lambda$. For a network $\mathcal{N} = (N,E)$, an $O(|N||E|\log(|N||E|))$ time algorithm to find all the $\lambda$-cent-dian points is proposed by \cite{hansen1991median}. Moreover, they also introduced the concept of \textit{generalized-center} as the minimizer of the difference function between the center and the median. In fact, when $\lambda \rightarrow \infty$, the ratio $\frac{H_{\lambda}}{\lambda}$ tends to the generalized-center. 
The generalized-center is an objective that favors the equity between O/D pairs. However, it could lead to inefficient solutions as noted by \cite{ogryczak1997centdians}. An axiomatic approach to the $\lambda$-cent-dian criterion was given in \cite{carrizosa1994axiomatic}.

The $\lambda$-cent-dian objective has also been considered in extensive facility location problems. A facility is called extensive if it is too large regarding its environment to be represented by isolated points. Examples of extensive facilities are paths, cycles, or trees on graphs and straight lines, circles, or hyperplanes in Euclidean spaces (see \cite{mesa1996review}, \cite{puerto2009extensions}).

The importance of the $\lambda$-cent-dian criterion allows weight, in some way, two contradicting criteria. Thus, the decision-maker chooses the weight to allocate to the center criterion and to the median criterion. To the best of our knowledge, we are the first to investigate and formalize the $\lambda$-cent-dian network design problem. In this context, our focus diverges from traditional facility location problems by identifying an optimal sub-network instead of single-point facilities and considering the demand given by origin-destination pairs (O/D pairs).

Network Design problems have been applied to several fields, especially in telecommunications and transportation systems (\cite{crainic2021network}). In some cases, the demand is not represented by individual points but rather by pairs of points (e.g., O/D pairs representing telephone calls), which produce flows that the network must manage. In transportation, the applications are diverse: air transportation, postal delivery systems, service networks, trucking, and transit systems. Often, there is only one network that handles the flows between origins and destinations. However, in some instances, more than one network is already in operation. In these situations, there is competition among the existing systems to capture demand. For instance, in mobile telephony, there is stiff competition among various providers. In urban and metropolitan mobility, a range of transportation modes, including private cars, bicycles, buses, and metros, compete for commuters' preferences.

Most Network Design problems primarily revolve around objective functions centered on either cost or profit. Given the inherent cost dependency on distance, these objective functions effectively serve as surrogates for the classical median problems, characterized by the summation of (weighted) distances to the facility. Furthermore, certain scenarios require additional considerations, where the proximity of origins and destinations plays a key role. For instance, in densely populated metropolitan areas, the daily commute of individuals to their workplaces necessitates minimal travel time, reflecting their reluctance to endure extended daily journeys between their residences and workplaces. This concern over commuting time can be viewed as a proxy for distance. Another example is electricity distribution generated by (solar, hydro, or wind) mini-power facilities in rural areas. In these cases, the electricity is provided at low voltage, and power losses increase with the distance (see \cite{gokbayrak2022two}). Thus, in these contexts, the center objective must be considered, combined with the median one.

The contributions of this work are the following:
\begin{itemize}
	\item We extend the solution concepts from the Facility Location Theory of the $\lambda$-cent-dian and generalized-center to the more intricate area of Network Design. 
	In this setting, we study the $\lambda$-cent-dian concept by considering the demand as a set of origin/destination pairs to be connected through a solution network. Furthermore, this solution network must satisfy a budget constraint. 
	The $\lambda$-cent-dian concept examines the balance between efficiency and equity. Notably, the generalized-center aligns with the specific case of the $\lambda$-cent-dian concept when $\lambda \rightarrow \infty$, emphasizing the variance between the measures of efficiency and equity.
    \item We delve into the properties of the $\lambda$-cent-dian and generalized-center solution networks, exploring equity, efficiency, and Pareto-optimality concerning the bicriteria center/median problem. Similar to findings in Facility Location, we ascertain that the Pareto-optimality solution set in the Network Design context is not always fully derived from minimizing the $\lambda$-cent-dian function for $\lambda \in [0,1]$. Examples illustrating these concepts are provided throughout the article for clarity.
 \item We prove that the problems are $NP$-hard by reducing the p-median and p-center problems to the corresponding network design problems. 
 \item We exemplify the potentiality of the concepts applied to real situations. For that, we consider the design of a metro network and that the potential riders can access the network through the pedestrian mode. We consider an upper bound for the pedestrian time/distance and a penalization for the pedestrian branches to refer to the differences between both modes.
\end{itemize}

The structure of the paper is as follows.
In Section \ref{sec:problem_definition}, we extend different concepts involving the median and center functions from Location Science to the more complex area of Network Design and analyze and compare the different solutions. Then, in Section \ref{sec:compromise_cent_dian} we present the relationship between the $\lambda$-cent-dians and Pareto optimal solutions and describe some peculiarities of them. Besides, in Section \ref{sec:complexity}, we prove that the problems addressed in this paper are NP-hard. Furthermore, in Section \ref{sec:applications} we show that these problems can be applied with minor modifications. Finally, our conclusions are presented in Section \ref{sec:conclusions}.

\section{Problems definition}\label{sec:problem_definition}

\subsection{Setting description}
In this paper, we address a network design problem using an underlying undirected graph to be $\mathcal{N} = (N,E)$, with associated costs for each node $i\in N$ and each link $e\in E$, denoted by $b_i \ge 0$ and $c_e \ge 0$, respectively. For every link $e=\{i,j\}\in E$, we define two arcs: $a=(i,j)$ and $\hat{a} =(j,i)$. We denote the set of resulting arcs as $A$. The length of each arc $a \in A$ is denoted by $d_a\geq 0$. We employ the notation $i\in e$ if node $i$ is a terminal node of $e$. Note that length can be substituted for time to traverse, generalized cost, or any other parameter assigned to each arc. We denote the set of edges incident to node $i$ by $\delta(i) \subseteq E$. Analogously, $\delta^+(i),\delta^-(i)\subseteq A$ denote the sets of arcs going out and in of node $i$. Let $d_{\mathcal{N}}(i,j)$ be the distance given by the shortest path between $i$ and $j$.

We assume that the mobility patterns are known and represented by a set $W\subset N\times N$ of origin/destination pairs (called O/D pairs), and a matrix collecting the expected demand between each O/D pair in a given period of time. Note that couples with equal origin and destination are not included in $W$. Specifically, for any given pair \( w = (w^s,w^t)\in W \) where \( w^s \neq w^t \), the demand traveling from the origin node \( w^s \in N \) to the destination node \( w^t\in N \) is known and denoted by \( g^w > 0 \). The aggregated demand is represented by \( G=\sum\limits_{w\in W}g^w \).

We also consider a private utility, \( u^w > 0 \), modeling that there already exists a (unique) mode, referred to as the private mode, to meet the demand. This mode competes with the potential network to be built on an all-or-nothing basis. In other words, an O/D pair \( (w^s,w^t) \) will utilize the new network only if it offers a path between \( w^s \) and \( w^t \) whose length or utility is equal to or shorter than the associated private utility \( u^w \). If these conditions are met, we say that the O/D pair is covered or served by the constructed network. Additionally, for each \( w\in W \), the subgraph \( \mathcal{N}^{w}=(N^w, E^w) \) comprises all nodes and edges that belong to a path in \( \mathcal{N} \) with a length less than or equal to \( u^w \). The corresponding set of arcs is represented by \( A^w \).

We are interested in subgraphs, denoted as $\mathcal{S}=(N_{\mathcal{S}},E_{\mathcal{S}})$, of $\mathcal{N}$ that can be constructed respecting a budget constraint. We represent this budget as a fraction $\alpha$ of the total cost of building the potential graph $\mathcal{N}$, noted by $C_{total}$. Therefore, a subgraph $\mathcal{S}$ is feasible if the condition
\begin{equation}
	\sum_{i\in N_{\mathcal S} } b_i + \sum_{e \in E_{\mathcal S}} c_e \le \alpha \left( \sum_{i\in N} b_i + \sum_{e \in E} c_e \right) = \alpha\, C_{total} \label{eq:subgraph_feasibility}
\end{equation}

\noindent is met. For a fixed value of $\alpha \in [0,1]$, we denote by $\mathbfcal{N}^\alpha$  the set of all subgraphs of $\mathcal N$ satisfying \eqref{eq:subgraph_feasibility}. 

For a given subgraph $\mathcal{S}$ and an O/D pair $w\in W$, the term $d_{\mathcal{S}}(w)$ represents the length of the shortest path from $w^s$ to $w^t$ within the subgraph $\mathcal{S}$. If it is not possible to connect a pair $w\in W$ within $\mathcal{S}$, we assume that $d_{\mathcal{S}}(w)=+\infty$. Taking this into account and given that each pair has an associated utility, each demand will travel from its origin to the destination on a path of length $\ell_{\mathcal{S}}(w)=\min\{d_{\mathcal{S}}(w),u^w\}$.

\subsection{Solution concepts}

We now introduce extensions of two solution concepts derived from location science, which are central to this work: the {\it median} and the {\it center}. For a given $\alpha\in (0,1)$, the (weighted) median is defined as a subgraph $\mathcal{S}_m$ in $\mathbfcal{N}^\alpha$ that minimizes the objective function
\begin{equation}
	F_m(\mathcal S) = \frac{1}{G} \sum_{w \in W} g^w\ell_{\mathcal{S}}(w). \label{median_obj}
\end{equation}
Given two subnetworks $\mathcal{S}_1,\mathcal{S}_2\subseteq\mathcal{N}$, it is said that $\mathcal{S}_1$ is more efficient that $\mathcal{S}_2$ iff $F_m(\mathcal{S}_1)<F_m(\mathcal{S}_2)$

A subgraph $\mathcal{S}_c$ is called a center if it is a minimizer of the following objective function
\begin{equation}
	F_c(S) = \max_{w \in W} \ell_{\mathcal{S}}(w). \label{center_obj}
\end{equation}

Related to the notion of center, it can be considered the \textit{weighted-center} solution. The \textit{weighted-center} solution is defined as the subgraph $\mathcal{S}_c^G$ in $\mathcal{N}^{\alpha}$ that minimizes the objective function

\begin{equation*}
	F_c^G(\mathcal{S}) = \frac{\max\limits_{w\in W}\left\{ g^w \ell_{\mathcal{S}}(w)\right\}}{g^{w_0}},\quad w_0=\argmax_{w\in W}\left\{ g^w\ell_{\mathcal{S}}(w)\right\}.
\end{equation*}

\noindent Given that the center minimizes the maximum travel time, it could lead to \textit{inefficient} solutions. This inefficiency is produced by not considering feasible subgraphs in which travel time decreases for some users maintaining the center objective function value. We depict this situation in Example 1.

\tikzstyle{vertex}=[circle,fill=white, draw=black,minimum size=20pt,font=\footnotesize,inner sep=0pt]
\tikzstyle{vertex2}=[circle,minimum size=15pt,font=\scriptsize,inner sep=0pt]
\tikzstyle{selected vertex} = [vertex, fill=red!24]
\tikzstyle{edge} = [draw,thick,-]
\tikzstyle{edge2} = [draw,dashed,-]
\tikzstyle{weight} = [font=\scriptsize,inner sep=1pt]
\tikzstyle{selected edge} = [draw,line width=5pt,-,red!50]
\tikzstyle{ignored edge} = [draw,line width=5pt,-,black!20]

\begin{example}
Consider the following network with a budget $\alpha\,C_{total} = 90$. Figure \ref{fig:example1} represents the underlying network where the couple labeling the edges represents building cost and length, respectively. Table \ref{table:example1} shows O/D pairs, private utility and demand of this example.
\begin{table}[ht] %
		\begin{minipage}{0.4\textwidth}
		\begin{tabular}{cccc}
			\hline
			Origin & Destination & $u^w$ & $g^w$ \\
			\hline
			1 & 6 & 92 & 200 \\
			2 & 5 & 92 & 50 \\
			4 & 1 & 92 & 50 \\
			\hline
		\end{tabular}\caption{Data in Example 1. We consider $\alpha\,C_{total} = 90$.}
		\label{table:example1}
	\end{minipage}
	\hspace*{0.5cm}
		\begin{minipage}{0.5\textwidth}
			\begin{tikzpicture}[scale=1.2, auto,swap]
				\foreach \pos/\name in {{(1.7,2.5)/1}, {(3,4)/2}, {(3,1)/3}, {(5,4)/4}, {(5,1)/5}, {(6.3,2.5)/6}}
				\node[vertex] (\name) at \pos {$\name$};
				\node[vertex2] at (1.2,2.5) {$5$};
				\node[vertex2] at (3,0.5) {$10$};
				\node[vertex2] at (3,4.5) {$10$};
				\node[vertex2] at (5,0.5) {$5$};
				\node[vertex2] at (5,4.5) {$5$};
				\node[vertex2] at (6.8,2.5) {$10$};
				\foreach \source/\dest/\weight in {2/1/{(20,30)}, 1/3/{(10,20)}, 3/2/{(10,10)}, 4/2/{(20,30)}, 3/4/{(40,60)}, 3/5/{(20,10)}, 5/4/{(10,30)}, 6/4/{(20,30)}, 5/6/{(30,70)}}
				\path[edge] (\source) -- node[weight] {$\weight$} (\dest);
			\end{tikzpicture}
			\captionsetup{hypcap=false}
			
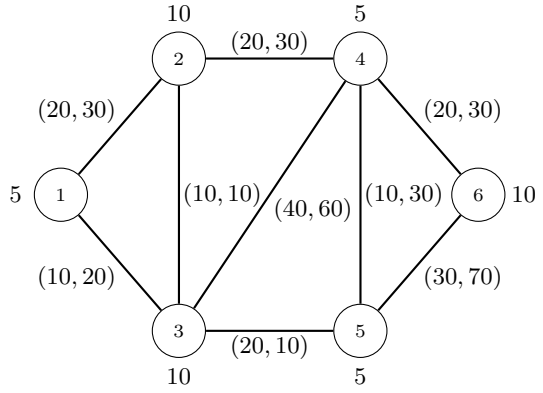
\captionof{figure}{Network in Example 1.}
			\label{fig:example1}
		\end{minipage}
	\end{table}

We observe that, with the given bound on the total cost, it is impossible to cover the three O/D pairs simultaneously. Hence, the objective value for the center is $F_c(\mathcal{S}_0)=92$. Note that the empty subgraph, $\mathcal{S}_0=\emptyset$, has a center value $F_c(\mathcal{S}_0) = 92$, and it is thus an optimal solution for the center problem. In other words, not constructing anything can be optimal for the center problem. Furthermore, this solution also yields a median value  $F_m(\mathcal{S}_0) = 92$. Now let us consider the following two subgraphs:

\begin{itemize}
	\item $\mathcal{S}_1$ composed by node set $N_1=\{1,2,4,6\}$ and edge set $E_1=\{(1,2),(2,4),(4,6)\}$ such that O/D pairs $(1,6)$ and $(4,1)$ are served; and 
	\item $\mathcal{S}_2$ with $N_2=\{1,2,4,5\}$ and $E_2=\{(1,2),(2,4),(4,5)\}$ such that O/D pairs $(2,5)$ and $(4,1)$ are served.
\end{itemize}
The solutions above have an objective function value for the center $F_c(\mathcal{S}_1) = F_c(\mathcal{S}_2) =92$, but the median objective changes. Indeed, these values are $F_m(\mathcal{S}_1)\approxeq 85.33$ and $F_m(\mathcal{S}_2)\approxeq 81.33$ respectively. 
\end{example}

\begin{example}
Using the aforementioned notation, we examine a network as outlined in Table \ref{table:example2} and illustrated in Figure \ref{fig:example2}. 
\begin{table}[ht]
	\begin{minipage}{0.4\textwidth}
		\begin{tabular}{cccc}
			\hline
			Origin & Destination & $u^w$ & $g^w$ \\
			\hline
			1 & 6 & 92 & 5 \\
			2 & 3 & 40 & 65 \\
			4 & 1 & 50 & 50 \\
			\hline
		\end{tabular}
		\caption{Data in Example 2. We consider $\alpha\,C_{total} = 90$.}
		\label{table:example2}
	\end{minipage}
	\hspace*{0.5cm}
	\begin{minipage}{0.5\textwidth}
		\begin{tikzpicture}[scale=1.2, auto,swap]
			\foreach \pos/\name in {{(1.7,2.5)/1}, {(3,4)/2}, {(3,1)/3}, {(5,4)/4}, {(5,1)/5}, {(6.3,2.5)/6}}
			\node[vertex] (\name) at \pos {$\name$};
			\node[vertex2] at (1.2,2.5) {$5$};
			\node[vertex2] at (3,0.5) {$10$};
			\node[vertex2] at (3,4.5) {$10$};
			\node[vertex2] at (5,0.5) {$5$};
			\node[vertex2] at (5,4.5) {$5$};
			\node[vertex2] at (6.8,2.5) {$10$};
			\foreach \source/\dest/\weight in {2/1/{(20,20)}, 1/3/{(10,30)}, 3/2/{(10,20)}, 4/2/{(20,25)}, 3/4/{(40,40)}, 3/5/{(20,30)}, 5/4/{(10,30)}, 6/4/{(30,70)}, 5/6/{(30,20)}}
			\path[edge] (\source) -- node[weight] {$\weight$} (\dest);
		\end{tikzpicture}
		\captionsetup{hypcap=false}
		
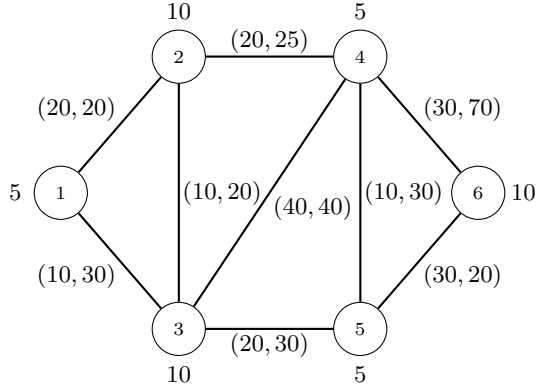
\captionof{figure}{Network in Example 2.}
		\label{fig:example2}
	\end{minipage}
\end{table}
In this example, the solution that optimizes the median does not serve the most distant pair or those with low demand. By contrast, when minimizing the center objective, a significantly larger value for the median objective emerges.

The optimal median and center subgraphs are illustrated in Figure \ref{fig:example2_1} and \ref{fig:example2_2}, respectively. 

\begin{minipage}{0.45\textwidth}
	\centering
	\begin{tikzpicture}[scale=0.5,auto,swap]
			\foreach \pos/\name in {{(1.5,2.5)/1}, {(3,4)/2}, {(3,1)/3}, {(5,4)/4}}
			\node[vertex] (\name) at \pos {$\name$};
			\node[vertex2] at (4,5.5) {\normalsize{$\mathcal{S}_m$}};
			\foreach \source/\dest in {2/1, 3/2, 4/2}
			\path[edge] (\source) -- (\dest);
		\end{tikzpicture}
	\captionsetup{hypcap=false}
	\captionof{figure}{Median network in Example 2.}
	\label{fig:example2_1}
\end{minipage}
\begin{minipage}{0.45\textwidth}
	\centering
	\begin{tikzpicture}[scale=0.5, auto,swap]
			\foreach \pos/\name in {{(1.5,2.5)/1}, {(3,1)/3}, {(5,1)/5}, {(6.5,2.5)/6}}
			\node[vertex] (\name) at \pos {$\name$};
			\node[vertex2] at (4,5.5) {\normalsize{$\mathcal{S}_c$}};
			\foreach \source/\dest in { 1/3, 3/5, 5/6}
			\path[edge] (\source) -- (\dest);
		\end{tikzpicture}
	\captionsetup{hypcap=false}
	\captionof{figure}{Center network in Example 2.}
	\label{fig:example2_2}
\end{minipage}
\end{example}

The computed values for the median and center objectives are as follows:
$$F_m(\mathcal{S}_m)=\frac{1}{120}(5\cdot 92 + 65\cdot 20 + 50\cdot 45)= \frac{4010}{120}\approxeq38.19,\quad F_c(\mathcal{S}_m)=92,$$
$$F_m(\mathcal{S}_c)=\frac{1}{120}(5\cdot 80 + 65\cdot 40 + 50\cdot 50)= \frac{5500}{120}\approxeq52.38, \quad F_c(\mathcal{S}_c)=80.$$
\noindent For the respective networks, it is noteworthy that even when considering the weighted-center, the loss in efficiency may still persist in specific scenarios. For instance, in this particular case, the optimal weighted-center remains $\mathcal{S}_c$ with a value of $F_c^G(\mathcal{S}_c) = 40$.

Examples 1 and 2 show the necessity of having a refined conceptual framework to consider solutions that are not dominated and capture the trade-off between both solution concepts. While the median generally prioritizes users located at the network's center, often to the detriment of those in distant regions, the center distinctly favors those in remote areas without necessarily considering the efficiency of the design. Recognizing this need for balanced solutions, researchers in the field of location science have, since the 1970s, explored the {\it $\lambda$-cent-dian} concept (refer to \citet{halpern1976location}). This concept represents a convex combination of both the center and median objectives. In the realm of network design, for a given $\lambda \in[0,1]$, the $\lambda$-cent-dian is a subgraph that aims to minimize the following objective function:
\begin{equation}
	H_\lambda(S) = \lambda F_c(\mathcal S) + (1-\lambda)F_m(\mathcal S) = \lambda \max_{w \in W} \ell_{\mathcal{S}}(w) + (1-\lambda) \frac{1}{G} \sum_{w \in W} g^w\ell_{\mathcal{S}}(w), \label{eq:cent-dian}
\end{equation}

As it is mentioned in \cite{ogryczak1997centdians}, the problem of finding the $\lambda$-cent-dian can be seen as the weighted version of finding the solution to the bi-criteria problem of minimizing objectives \eqref{median_obj} and \eqref{center_obj}. 

\cite{hansen1991median} introduced the concept of the {\it generalized-center} for facility location. This concept was formulated to reduce discrepancies in accessibility among users as much as possible. In this work, the {\it generalized-center} corresponds to a subgraph, $\mathcal S_{gc}$, which minimizes the disparity between the center and median objectives, denoted as $F_{gc}(\mathcal{S})$. When one considers the difference between the function $F_m(\mathcal{S})$ and the function $F_c^G(\mathcal{S})$, instead of $F_c(\mathcal{S})$, the resultant solution network is named the \textit{weighted-generalized-center}. The objective function for this is denoted by $F_{gc}^G(\mathcal{S})$.

As we have observed in Example 1, one generalized-center is the empty subgraph, which is an unreasonable solution from the point of view of the median value. Example 3 exemplifies a situation where the generalized-center worsens both the center and the median values. To avoid such solution networks, the optimal solution network is restricted to the set of solution networks so-called as \textit{Pareto-optimal concerning the distances (of the shortest paths)}. A subgraph $\mathcal S \in \mathbfcal{N}^\alpha$ is {Pareto-optimal with respect to the distances (travel times)} if there does not exist a subgraph $\mathcal S' \in \mathbfcal{N}^\alpha$ such that $\ell_{\mathcal{S}'}(w) \le \ell_{\mathcal{S}}(w),$
for all $w\in W$ where at least one of these inequalities is strictly satisfied. We denote the set of subgraphs satisfying this definition of Pareto-optimality as $\mathbfcal{PO}^\alpha$. Thus, we redefine the generalized-center as the optimal solution to the problem
\begin{equation}\label{eq:GC}
	\min\{F_c(\mathcal{S})-F_m(\mathcal{S}): \mathcal{S}\in \mathbfcal{PO}^{\alpha}\}.
\end{equation}
In the same way, the \textit{weighted-generalized-center} is defined as the optimal solution to the problem
\begin{equation}\label{eq:GC_weighted}
	\min\{|F^G_c(\mathcal{S})-F_m(\mathcal{S})|: \mathcal{S}\in \mathbfcal{PO}^{\alpha}\}.
\end{equation}
The following examples depict the discussion above.

\begin{example}
Let us consider the following example described in Table \ref{table:example3} and Figure \ref{fig:example3}.

The generalized-center $\mathcal{S}_{gc}$ is:

\vspace{0.3cm}
\begin{minipage}{0.25\textwidth}
	\begin{tikzpicture}[scale=0.45, auto,swap]
		\foreach \pos/\name in {{(3,4)/2}, {(3,1)/3}, {(5,4)/4}, {(5,1)/5}, {(6.5,2.5)/6}}
		\node[vertex] (\name) at \pos {$\name$};
		\foreach \source/\dest in {2/3, 3/5, 5/4, 4/6}
		\path[edge] (\source) -- (\dest);
	\end{tikzpicture}
\end{minipage}
\begin{minipage}{0.65\textwidth}
	$F_c(\mathcal{S}_{gc})-F_m(\mathcal{S}_{gc}) =92- \frac{1}{105}(50\cdot 92+5\cdot 70+50\cdot 92) \approxeq 92-90.95$.
\end{minipage}
\vspace{0.3cm}

\noindent Nevertheless, there exists a network, $\mathcal{S}_1$, more efficient in terms of the distances of the shortest paths than $\mathcal{S}_{gc}$:

\begin{minipage}{0.25\textwidth}
	\begin{tikzpicture}[scale=0.45, auto,swap]
		\foreach \pos/\name in {{(1.5,2.5)/1}, {(3,4)/2}, {(5,4)/4}, {(6.5,2.5)/6}}
		\node[vertex] (\name) at \pos {$\name$};
		\foreach \source/\dest/\weight in {2/1, 4/2, 6/4}
		\path[edge] (\source) -- (\dest);
	\end{tikzpicture}
\end{minipage}
\begin{minipage}{0.65\textwidth}
	$F_c(\mathcal{S}_1)-F_m(\mathcal{S}_1) = 40 - \frac{1}{105}(50\cdot 30+ 5\cdot 20+50\cdot 40) \approxeq 40-34.28$
\end{minipage}
\vspace{0.5cm}

In fact, $\mathcal{S}_1$ is simultaneously the center and the median. Even if we minimize $F_{gc}^G(\,\cdot\,)$, instead of $F_{gc}(\,\cdot\,)$, the optimal solution does not change, which is $\mathcal{S}_{gc}$. Network $\mathcal{S}_1$ is Pareto-optimal with respect to the distances, but $\mathcal{S}_{gc}$ is not. Even more, each of the O/D pairs has its shortest path shorter in $\mathcal{S}_1$ than in $\mathcal{S}_{gc}$. Locating $\mathcal{S}_{gc}$ can worsen the center and median values. This solution does not capture the trade-off between the center and the median values.
\end{example}

\begin{table}[ht]
\begin{minipage}{0.4\textwidth}
		\begin{tabular}{cccc}
			\hline
			Origin & Destination & $u^w$ & $g^w$ \\
			\hline
			1 & 2 & 92 & 50 \\
			2 & 6 & 100 & 5 \\
			4 & 1 & 92 & 50 \\
			\hline
		\end{tabular}
		\caption{Data in Example 3. We consider $\alpha\,C_{total} = 90$.}
		\label{table:example3}
\end{minipage}
\hspace*{0.5cm}
\begin{minipage}{0.5\textwidth}
	\begin{tikzpicture}[scale=1.2, auto,swap]
		\foreach \pos/\name in {{(1.7,2.5)/1}, {(3,4)/2}, {(3,1)/3}, {(5,4)/4}, {(5,1)/5}, {(6.3,2.5)/6}}
		\node[vertex] (\name) at \pos {$\name$};
		\node[vertex2] at (1.2,2.5) {$5$};
		\node[vertex2] at (3,0.5) {$10$};
		\node[vertex2] at (3,4.5) {$10$};
		\node[vertex2] at (5,0.5) {$5$};
		\node[vertex2] at (5,4.5) {$5$};
		\node[vertex2] at (6.8,2.5) {$10$};
		\foreach \source/\dest/\weight in {2/1/{(20,30)}, 1/3/{(10,20)}, 3/2/{(10,20)}, 4/2/{(20,10)}, 3/4/{(70,60)}, 3/5/{(5,10)}, 5/4/{(10,30)}, 6/4/{(20,10)}, 5/6/{(10,30)}}
		\path[edge] (\source) -- node[weight] {$\weight$} (\dest);
	\end{tikzpicture}
	\captionsetup{hypcap=false}
	
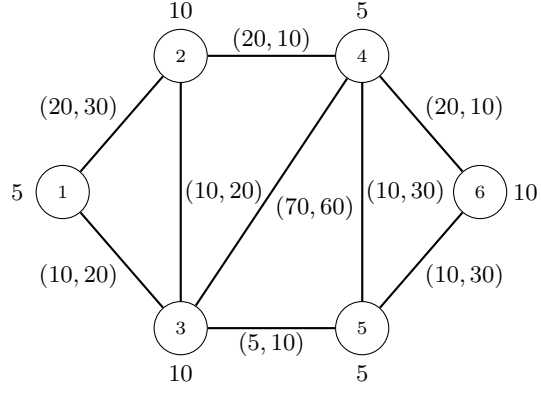
\captionof{figure}{Network in Example 3.}
	\label{fig:example3}
\end{minipage}
\end{table}

Even when the problem of the generalized-center is constrained to the set $\mathbfcal{PO}^{\alpha}$, there is potential for obtaining highly inefficient solution graphs as optimal solutions. This is illustrated in Example $4$. Specifically, when minimizing the difference $F_c(\,\cdot\,)-F_m(\,\cdot\,)$, if multiple solution networks have the same $F_c(\,\cdot\,)$ value, the network with the least favorable $F_m(\,\cdot\,)$ value will be chosen.

\begin{example}
We depict in Table \ref{table:example4} and Figure \ref{fig:example4} a situation where the generalized-center in $\mathbfcal{PO}^{\alpha}$ leads to a very inefficient solution.

\begin{table}[ht]
\begin{minipage}{0.4\textwidth}
\begin{tabular}{cccc}
			\hline
			Origin & Destination & $u^w$ & $g^w$ \\
			\hline
			1 & 2 & 35 & 50 \\
			2 & 4 & 35 & 30 \\
			3 & 1 & 35 & 30 \\
			4 & 3 & 35 & 20 \\
			\hline
		\end{tabular}
		\caption{Data in Example 4. We consider $\alpha\,C_{total} = 50$.}
		\label{table:example4}
\end{minipage}
\hspace*{0.5cm}
\begin{minipage}{0.43\textwidth}
	\centering
	\begin{tikzpicture}[scale=1, auto,swap]
		\foreach \pos/\name in {{(3,4)/1}, {(3,1)/3}, {(6,4)/2}, {(6,1)/4}}
		\node[vertex] (\name) at \pos {$\name$};
		\node[vertex2] at (3,0.5) {$5$};
		\node[vertex2] at (3,4.5) {$5$};
		\node[vertex2] at (6,0.5) {$5$};
		\node[vertex2] at (6,4.5) {$5$};
		\foreach \source/\dest/\weight in {2/1/{(10,10)}, 1/3/{(10,10)}, 3/4/{(10,10)}, 4/2/{(10,10)} }
		\path[edge] (\source) -- node[weight] {$\weight$} (\dest);
	\end{tikzpicture}
	\captionsetup{hypcap=false}
	\captionof{figure}{Network in Example 4.}
	\label{fig:example4}
\end{minipage}
	\end{table}

The set $\mathbfcal{PO}^{\alpha}$ consists of:
\vspace{0.3cm}

\begin{minipage}{0.20\textwidth}
	\begin{tikzpicture}[scale=0.4, auto,swap]
		\foreach \pos/\name in {{(3,4)/1}, {(3,1)/3}, {(6,4)/2}, {(6,1)/4}}
		\node[vertex] (\name) at \pos {$\name$};
		\node[vertex2] at (1.7,2.5) {\normalsize{$\mathcal{S}_1$}};
		\foreach \source/\dest in {1/2, 1/3, 3/4}
		\path[edge] (\source) -- (\dest);
	\end{tikzpicture}
\end{minipage}
\begin{minipage}{0.70\textwidth}
	\begin{equation*}
	\begin{array}{l}
\displaystyle F_c(\mathcal{S}_1)-F_m(\mathcal{S}_1) =\\ 
\displaystyle =30 - \frac{1}{130}(50\cdot 10 + 30\cdot30 + 30\cdot 10 + 20\cdot 10) =\\ \noalign{\smallskip}
\displaystyle = 30 - \frac{1900}{130} \approxeq 15.38	
	\end{array}
	\end{equation*}
\end{minipage}
\vspace{0.5cm}

\begin{minipage}{0.20\textwidth}
	\begin{tikzpicture}[scale=0.4, auto,swap]
		\foreach \pos/\name in {{(3,4)/1}, {(3,1)/3}, {(6,4)/2}, {(6,1)/4}}
		\node[vertex] (\name) at \pos {$\name$};
		\node[vertex2] at (1.7,2.5) {\normalsize{$\mathcal{S}_2$}};
		\foreach \source/\dest in {1/3, 2/4, 3/4}
		\path[edge] (\source) -- (\dest);
	\end{tikzpicture}
\end{minipage}
\begin{minipage}{0.70\textwidth}
		\begin{equation*}
		\begin{array}{l}
F_c(\mathcal{S}_2)-F_m(\mathcal{S}_2) =\\
\displaystyle = 30 - \frac{1}{130}(50\cdot 30 + 30\cdot10 + 30\cdot 10 + 20\cdot 10) =\\ \noalign{\smallskip}
\displaystyle = 30 - \frac{2300}{130} \approxeq 12.30
	\end{array}
\end{equation*}
\end{minipage}
\vspace{0.5cm}

\begin{minipage}{0.20\textwidth}
	\begin{tikzpicture}[scale=0.4, auto,swap]
		\foreach \pos/\name in {{(3,4)/1}, {(3,1)/3}, {(6,4)/2}, {(6,1)/4}}
		\node[vertex] (\name) at \pos {$\name$};
		\node[vertex2] at (1.7,2.5) {\normalsize{$\mathcal{S}_3$}};
		\foreach \source/\dest in {1/2, 2/4, 3/4}
		\path[edge] (\source) -- (\dest);
	\end{tikzpicture}
\end{minipage}
\begin{minipage}{0.70\textwidth}
\begin{equation*}
\begin{array}{l}
	F_c(\mathcal{S}_3)-F_m(\mathcal{S}_3) =\\
\displaystyle = 30 - \frac{1}{130}(50\cdot 10 + 30\cdot10 + 30\cdot 30 + 20\cdot 10) = \\ \noalign{\smallskip}
\displaystyle = 30 - \frac{1900}{130} \approxeq 15.38
\end{array}
\end{equation*}
\end{minipage}
\vspace{0.5cm}

\begin{minipage}{0.20\textwidth}
	\begin{tikzpicture}[scale=0.4, auto,swap]
		\foreach \pos/\name in {{(3,4)/1}, {(3,1)/3}, {(6,4)/2}, {(6,1)/4}}
		\node[vertex] (\name) at \pos {$\name$};
		\node[vertex2] at (1.7,2.5) {\normalsize{$\mathcal{S}_4$}};
		\foreach \source/\dest in {1/2, 1/3, 2/4}
		\path[edge] (\source) -- (\dest);
	\end{tikzpicture}
\end{minipage}
\begin{minipage}{0.70\textwidth}
	\begin{equation*}
	\begin{array}{l}
F_c(\mathcal{S}_4)-F_m(\mathcal{S}_4) =\\
\displaystyle = 30 - \frac{1}{130}(50\cdot 10 + 30\cdot10 + 30\cdot 10 + 20\cdot 30) =\\ \noalign{\smallskip}
\displaystyle = 30 - \frac{1430}{130} \approxeq 19
	\end{array}
	\end{equation*}
\end{minipage}
\vspace{0.5cm}

The four solution networks belong to $\mathbfcal{PO}^{\alpha}$, and each one is a center. While $\mathcal{S}_2$ is the generalized-center, $\mathcal{S}_4$ is the most efficient network. One might assume that the issues highlighted with the median and center values are resolved when considering the weighted function. However, this is not the case. It is easy to check that the generalized-weighted-center remains $\mathcal{S}_2$.
\end{example}

The last example shows the necessity to revisit the notion of Pareto-optimality. We perform this by extending the notion of Pareto-optimality to the bi-criteria setting: a subgraph $\mathcal{S}$ is Pareto-optimal if there no exists another subgraph $\mathcal{S'}$ in $\mathbfcal{N}^\alpha$ for which 
\begin{equation}
	F_m(S') \le F_m(S) \quad \mbox{ and } \quad F_c(S') \le F_c(S),
\end{equation}
\noindent with one of the inequalities being strict. We denote the set of subgraphs satisfying this property $\mathbfcal{PO}^\alpha_2$. Hence, we define the \textit{restricted-generalized-center} as an optimal solution to the problem 
\begin{equation}
	\min\{F_c(\mathcal{S})-F_m(\mathcal{S}): \mathcal{S}\in \mathbfcal{PO}^{\alpha}_2\}. \label{eq:restricted_general_center}
\end{equation}

\begin{remark}\label{remark:generalized_center}
	In Example 4, $\mathbfcal{PO}^{\alpha}_2$ is composed only of $\mathcal{S}_4$, then it is the unique minimizer of \eqref{eq:restricted_general_center}. Furthermore, it is easy to check that always $\mathbfcal{PO}^{\alpha}_2\subseteq \mathbfcal{PO}^{\alpha}$.
\end{remark}

Analogously to the restricted-generalized-center, we introduce the \textit{$\lambda$-restricted-cent-dian}. It is defined as the optimal solution to the problem:
\begin{equation}
	\min\{H_{\lambda}(\mathcal{S}): \mathcal{S}\in \mathbfcal{PO}^{\alpha}_2\}.
\end{equation}

Then, by construction, it is easy to check that for $\lambda\in(0,1)$ the $\lambda$-restricted-cent-dian is simply a $\lambda$-cent-dian. Furthermore, the converse is also true.

\begin{proposition}\label{remark:centdian_pO2} For $\lambda\in(0,1)$, the corresponding $\lambda$-cent-dian always belongs to the set $\mathbfcal{PO}^{\alpha}_2$. 
\end{proposition}
\begin{proof}
This can be easily proved by \textit{Reductio ad absurdum}. Let's assume that $\mathcal{S}$ is the $\lambda$-cent-dian but $\mathcal{S}\notin \mathcal{PO}_2^{\alpha}$. Then, since $\mathcal{S}\notin \mathcal{PO}_2^{\alpha}$ then there exists $\mathcal{S}'$ such that $F_m(\mathcal{S}')\leq F_m(\mathcal{S})$ and $F_c(\mathcal{S}')\leq F_c(\mathcal{S})$, with one of the inequalities being strict. Hence $\lambda F_m(\mathcal{S}')+(1-\lambda) F_c(\mathcal{S}')< F_m(\mathcal{S})+(1-\lambda) F_c(\mathcal{S})$. That is, $\mathcal{S}$ is not a $\lambda$-cent-dian, with is a contradiction. Thus, in this case, the $\lambda$-cent-dian is always a $\lambda$-restricted-cent-dian. 
\end{proof}

The statement of Proposition \ref{remark:centdian_pO2} is not necessarily true for $\lambda\geq1$, specifically for the limiting case $\lambda\rightarrow\infty$. As shown in Example 4, the generalized-center does not necessarily belong to the set $\mathbfcal{PO}^{\alpha}_2$. Besides, we will show that for all $\lambda\geq 1$ the corresponding $\lambda$-restricted-cent-dian is always a center. Note that, given a bound $\alpha$, the center of a general network $\mathcal{N}$ may be non-unique and, in such cases, not all centers belong to $\mathbfcal{PO}^{\alpha}_2$. This issue is shown in Example $4$. We know that $\mathbfcal{PO}^{\alpha}=\{\mathcal{S}_1,\mathcal{S}_2,\mathcal{S}_3,\mathcal{S}_4\}$, but $\mathcal{S}_1,\mathcal{S}_2,\mathcal{S}_3\notin\mathbfcal{PO}^{\alpha}_2$. To belong to $\mathbfcal{PO}^{\alpha}_2$, the center must be unique or it must be the center with the best value of $F_m(\,\cdot\,)$. Thus, the center belonging to $\mathbfcal{PO}^{\alpha}_2$ is an optimal solution of the following lexicographic problem
\begin{equation}\label{eq:lexmin}
	\mbox{lex}\min\{[F_c(\mathcal{S}),F_m(\mathcal{S})]:\,\,\mathcal{S}\in\mathbfcal{N}^{\alpha}\}.
\end{equation}

The lexicographic minimization in \eqref{eq:lexmin} entails that we first minimize $F_c(\,\cdot\,)$ on the set $\mathbfcal{N}^{\alpha}$. Subsequently, we minimize $F_m(\,\cdot\,)$ on the optimal subnetworks set determined by $F_c(\,\cdot\,)$. This second minimization step is only necessary when the optimal solution of $F_c(\,\cdot\,)$ is not unique. We refer to the optimal solution of \eqref{eq:lexmin} as a \textit{lexicographic cent-dian}. It is important to highlight that the lexicographic cent-dian is also a center, and in the case where the center is unique, this center is also the lexicographic cent-dian.

\begin{proposition}\label{prop:restricted_lexmin}
	Consider $\lambda \ge 1.$ The restricted-generalized-center and the $\lambda$-restricted-cent-dian are lexicographic cent-dian in $\mathcal{N}^{\alpha}$.
	Conversely, any lexicographic cent-dian is a restricted generalized center and a $\lambda$-restricted-cent-dian.
\end{proposition}

\begin{proof}
	Let $\mathcal{S}$ be the lexicographic cent-dian. That is, $F_c(\mathcal{S})\leq F_c(\mathcal{S}')$ for any $\mathcal{S}'\in\mathbfcal{N}^{\alpha}$, or, if $\exists\,\mathcal{S}'\in\mathbfcal{N}^{\alpha}$ such that $F_c(\mathcal{S})=F_c(\mathcal{S}')$, then $F_m(\mathcal{S})\leq F_m(\mathcal{S}')$.
	Firstly, observe that $\mathcal{S}\in \mathbfcal{PO}_2^{\alpha}$. To continue, we consider $\mathcal{S}'\in \mathbfcal{N}^{\alpha}$ a solution network of $\mathbfcal{PO}_2^{\alpha}$. If $\mathcal{S}'$ is not a lexicographic cent-dian, then $F_c(\mathcal{S})< F_c(\mathcal{S}')$ or, if $F_c(\mathcal{S})= F_c(\mathcal{S}')$ then $F_m(\mathcal{S})< F_m(\mathcal{S}')$. This second situation is not possible since $\mathcal{S}'\in\mathbfcal{PO}_2^{\alpha}$. Hence, being in $\mathbfcal{PO}_2^{\alpha}$, $\mathcal{S}'$ has to satisfy inequalities 
	$$F_c(\mathcal{S})< F_c(\mathcal{S}') \quad \text{ and } \quad F_m(\mathcal{S})> F_m(\mathcal{S}').$$
	Thus, $F_c(\mathcal{S}')-F_m(\mathcal{S}')>F_c(\mathcal{S})-F_m(\mathcal{S})$, which means that the restricted-generalized-center is a lexicographic cent-dian. Aditionally, for $\lambda\geq 1$, this property holds:
	$$H_{\lambda}(\mathcal{S}')=F_c(\mathcal{S}')+(\lambda-1)(F_c(\mathcal{S}')-F_m(\mathcal{S}'))>F_c(\mathcal{S})+(\lambda-1)(F_c(\mathcal{S})-F_m(\mathcal{S}))=H_{\lambda}(\mathcal{S}),$$
	which proves that the corresponding $\lambda$-restricted-cent-dian is a lexicographic cent-dian. 
	
	To finish the proof, since all lexicographic cent-dians have the same center value and median value, they are all restricted-generalized-centers and $\lambda$-restricted-cent-dians with $\lambda>1$.
\end{proof}

Proposition \ref{prop:restricted_lexmin} provides us with a very simple characteristic of the $\lambda$-restricted-cent-dian for $\lambda\geq1$. From this, we conclude that the $\lambda$-restricted-cent-dians for $\lambda\geq1$ are simply the centers with the best median value. On the other hand, it means that the solution concept of the restricted-generalized-center does not provide us with any compromise between the median and center values.

\section{On the relationship of \texorpdfstring{$\lambda$}{}-cent-dians and Pareto optimal solutions}\label{sec:compromise_cent_dian}

Pareto optimal solutions, denoted as $\mathbfcal{PO}^{\alpha}_2$, offer a set of sub-networks that efficiently strike a balance between spatial efficiency and equity. This is achieved by compromising between the median and center objectives. On the other hand, the concept of $\lambda$-cent-dian seeks to find this balance through a linear combination of both objectives. However, the sub-networks produced by the $\lambda$-cent-dian do not encompass all the sub-networks in $\mathbfcal{PO}^{\alpha}_2$.

Indeed, Example 5 illustrates the fact that networks containing cycles might have sub-networks that belong to $\mathbfcal{PO}^{\alpha}_2$ but differ from $\mathcal{S}_c$, $\mathcal{S}_m$ and any $\lambda$-cent-dian for $\lambda \in [0,1]$.

\begin{example}
Consider the potential network depicted in Table \ref{table:example5} and Figure \ref{fig:example5}, with the fixed bound $\alpha\,C_{total} = 70$.
\begin{table}[ht]
\begin{minipage}{0.4\textwidth}
\begin{tabular}{cccc}
\hline
Origin & Destination & $u^w$ & $g^w$ \\
\hline
1 & 2 & 70 & 90 \\
3 & 4 & 55 & 20 \\
5 & 6 & 92 & 5 \\
\hline
\end{tabular}\caption{Data in Example 5. We consider $\alpha\,C_{total} = 70$.}
\label{table:example5}
\end{minipage}
\hspace*{0.5cm}
\begin{minipage}{0.5\textwidth}
	\begin{tikzpicture}[scale=1.2, auto,swap]
		\foreach \pos/\name in {{(1.7,2.5)/1}, {(3,4)/2}, {(3,1)/3}, {(5,4)/4}, {(5,1)/5}, {(6.3,2.5)/6}}
		\node[vertex] (\name) at \pos {$\name$};
		\node[vertex2] at (1.2,2.5) {$5$};
		\node[vertex2] at (3,0.5) {$10$};
		\node[vertex2] at (3,4.5) {$10$};
		\node[vertex2] at (5,0.5) {$5$};
		\node[vertex2] at (5,4.5) {$5$};
		\node[vertex2] at (6.8,2.5) {$10$};
		\foreach \source/\dest/\weight in {2/1/{(55,10)}, 1/3/{(5,20)}, 3/2/{(5,40)}, 4/2/{(50,10)}, 3/4/{(5,35)}, 3/5/{(50,10)}, 5/4/{(5,40)}, 6/4/{(5,40)}, 5/6/{(55,20)}}
		\path[edge] (\source) -- node[weight] {$\weight$} (\dest);
	\end{tikzpicture}
	\captionsetup{hypcap=false}
	\captionof{figure}{Network in Example 5.}
	\label{fig:example5}
\end{minipage}
\end{table}

The median $\mathcal{S}_m$ is

\vspace*{0.3cm}
\begin{minipage}{0.25\textwidth}
	\begin{tikzpicture}[scale=0.45, auto,swap]
		\foreach \pos/\name in {{(1.5,2.5)/1}, {(3,4)/2}}
		\node[vertex] (\name) at \pos {$\name$};
		\foreach \source/\dest in {2/1}
		\path[edge] (\source) -- (\dest);
	\end{tikzpicture}
\end{minipage}
\begin{minipage}{0.65\textwidth}
	$F_m(\mathcal{S}_m) = \frac{1}{115}(90\cdot 10 + 20\cdot55 + 5\cdot 92) = \frac{2460}{115} \approxeq 21.39$
	
	$F_c(\mathcal{S}_m)= 92$
\end{minipage}
\vspace{0.3cm}

The network has a unique center $\mathcal{S}_c$:

\vspace{0.3cm}
\begin{minipage}{0.25\textwidth}
	\begin{tikzpicture}[scale=0.45, auto,swap]
		\foreach \pos/\name in {{(1.5,2.5)/5}, {(3,4)/6}}
		\node[vertex] (\name) at \pos {$\name$};
		\foreach \source/\dest in {5/6}
		\path[edge] (\source) -- (\dest);
	\end{tikzpicture}
\end{minipage}
\begin{minipage}{0.65\textwidth}
	$F_m(\mathcal{S}_c) = \frac{1}{115}(90\cdot 70 + 20\cdot55 + 5\cdot 20) = \frac{7500}{115} \approxeq 65.21$
	
	$F_c(\mathcal{S}_c)= 70$
\end{minipage}
\vspace{0.3cm}

There is another subnetwork belonging to $\mathbfcal{PO}^{\alpha}_2$. The following network $\mathcal{S}_1$ has a better value for the average of the shortest paths than $\mathcal{S}_c$ and its center value is lower than in $\mathcal{S}_m$.

\vspace{0.3cm}
\begin{minipage}{0.25\textwidth}
	\begin{tikzpicture}[scale=0.45, auto,swap]
		\foreach \pos/\name in {{(1.5,2.5)/1}, {(3,4)/2}, {(3,1)/3}, {(5,4)/4}, {(5,1)/5}, {(6.5,2.5)/6}}
		\node[vertex] (\name) at \pos {$\name$};
		\foreach \source/\dest in {1/3, 3/2, 3/4, 5/4, 6/4}
		\path[edge] (\source) -- (\dest);
	\end{tikzpicture}
\end{minipage}
\begin{minipage}{0.65\textwidth}
	$F_m(\mathcal{S}_1) = \frac{1}{115}(90\cdot 60 + 20\cdot35 + 5\cdot 80) = \frac{6500}{115} \approxeq 56.52$
	
	$F_c(\mathcal{S}_1)= 80$
\end{minipage}
\vspace{0.3cm}

We can check that for any $0\leq\lambda\leq1$, $H_{\lambda}(\mathcal{S}_m)<H_{\lambda}(\mathcal{S}_1)$ or $H_{\lambda}(\mathcal{S}_c)<H_{\lambda}(\mathcal{S}_1)$. That is, even though $\mathcal{S}_1\in\mathbfcal{PO}^{\alpha}_2$, it cannot be an $\lambda$-cent-dian for any $0\leq\lambda\leq1$. This issue remains even if we use the weighted function $F_c^G(\,\cdot\,)$. We obtain that the weighted-center coincides with $\mathcal{S}_m$, since $F_c^G(\mathcal{S}_m)=55$, $F_c^G(\mathcal{S}_c)=70$ and $F_c^G(\mathcal{S}_1)=60$. The same solution is obtained whatever the value of $\lambda\in[0,1]$ is. It does not result in any compromise between the value of the parameter $\lambda$ and the functions used. That is, the solution does not change according to the different values that the parameter $\lambda$ takes.
\end{example}

The scenario presented in Example $5$ is infeasible when considering a tree network, as demonstrated in Proposition \ref{prop:tree}. In other words, when dealing with a tree network, such a compromise is achievable. In Proposition \ref{prop:tree}, without loss of generality, we refer to $\mathcal{S}_c$ as either the sole center or the center possessing the optimal value of $F_m(\cdot)$.

\begin{remark}
Note that the Pareto-optimal set of solutions $\mathcal{PO}_2^{\alpha}$ is composed of the solutions known in the literature as supported and non-supported. For example, see \cite{amorosi2022two}, \cite{pozo2024biobjective} and \cite{przybylski2010recursive}.
\end{remark}

\begin{proposition}\label{prop:tree}
	On a tree network $\mathcal{T}$, the set $\mathbfcal{PO}^{\alpha}_2$ is composed only by $\mathcal{S}_m$ and $\mathcal{S}_c$.
\end{proposition}
\begin{proof} 
	Let $w_0\in W$ be the O/D pair which satisfies that $w_0=\arg\max\limits_{w\in W}\{\min\{d_{\mathcal{T}}(w),u^w\}\}$. We will assume that $d_{\mathcal{T}}(w_0)< u^w$ and $(\tilde{N}^{w_0},\tilde{E}^{w_0})\in\mathbfcal{T}^{\alpha}$ (its associated path is a feasible graph). Let us prove by \textit{Reductio ad absurdum}. Thus, let us suppose that there exists $\mathcal{S}\in \mathbfcal{PO}^{\alpha}_2$ such that $\mathcal{S}\neq\mathcal{S}_c$ and $\mathcal{S}\neq\mathcal{S}_m$. Therefore, $F_m(\mathcal{S}_m)<F_m(\mathcal{S})<F_m(\mathcal{S}_c)$ and $F_c(\mathcal{S}_c)<F_c(\mathcal{S})<F_c(\mathcal{S}_m)$, since if $F_m(\mathcal{S})=F_m(\mathcal{S}_c)$ or $F_c(\mathcal{S})=F_c(\mathcal{S}_m)$ then $\mathcal{S}\notin \mathbfcal{PO}^{\alpha}_2$.
	With regard to $F_c(\mathcal{S}_c)<F_c(\mathcal{S})<F_c(\mathcal{S}_m)$, it means that 
	$$\min\{d_{\mathcal{S}_c}(w_0),u^{w_0}\}< \min\{d_{\mathcal{S}}(w_0),u^{w_0}\} < \min\{d_{\mathcal{S}_m}(w_0),u^{w_0}\}.$$
	In order to satisfy the previous inequality, we have to set $$\min\{d_{\mathcal{S}_c}(w_0),u^{w_0}\}=d_{\mathcal{S}_c}(w_0) \text{ and }
	\min\{d_{\mathcal{S}}(w_0),u^{w_0}\}=d_{\mathcal{S}}(w_0).$$ Then,
	$$d_{\mathcal{T}}(w_0)\leq d_{\mathcal{S}_c}(w_0)< d_{\mathcal{S}}(w_0) < \min\{d_{\mathcal{S}_m}(w_0),u^{w_0}\},$$
	which is not possible since in a tree network the path for each O/D pair $w\in W$ is unique since there are no cycles. 
	Note that if $d_{\mathcal{T}}(w_0)\geq u^w$, then $F_c(\mathcal{S}_c)=F_c(\mathcal{S})=F_c(\mathcal{S}_m)=u^{w_0}$. Besides, since $F_m(\mathcal{S}_m)<F_m(\mathcal{S})$ and $F_m(\mathcal{S}_m)<F_m(\mathcal{S}_c)$, then $\mathcal{S}, \mathcal{S}_c\notin\mathbfcal{PO}^{\alpha}_2$.
\end{proof}

As in \cite{ogryczak1997centdians}, with the purpose of identifying some compromise $\lambda$-cent-dians on a general network, we need a solution concept different from the one discussed so far. As explained for the case of a nonconvex problem in Location Theory (see \cite{steuer1986multiple}), the set of Pareto-optimality concerning the bi-criteria center/median objective can be completely parametrized through the minimization of the weighted Chebychev norm. In the Network Design area, the Pareto-optimality set $\mathbfcal{PO}^{\alpha}_2$ can be completely parametrized through the minimization of the weighted function
\begin{equation}\label{eq:chevfunc}
	\Bar{H}_{\lambda}(\mathcal{S}) = \max\{\lambda\,F_c(\mathcal{S}),(1-\lambda)\,F_m(\mathcal{S})\}.
\end{equation}

In the case of a non-unique optimal solution, this optimization has to be subject to a second stage. Thus, we call a subgraph $\mathcal{S}\in\mathbfcal{N}^{\alpha}$ a \textit{maximum $\lambda$-cent-dian} if it is an optimal solution of the following lexicographic (two-level) problem
\begin{equation}\label{eq:lexmin2}
	\mbox{lex}\min\{[\Bar{H}_{\lambda}(\mathcal{S}),H_{\lambda}(\mathcal{S})]:\,\,\mathcal{S}\in\mathbfcal{N}^{\alpha}\}.
\end{equation}

The lexicographic minimization in \eqref{eq:lexmin2} means that first we minimize $\Bar{H}_{\lambda}(\,\cdot\,)$ on set $\mathbfcal{N}^{\alpha}$, and then we minimize $H_{\lambda}(\,\cdot\,)$ on the optimal subnetworks set of function $\Bar{H}_{\lambda}(\, \cdot \,)$. Thus, function $H_{\lambda}(\,\cdot\,)$ is minimized only for regulation purposes, in the case of a nonunique minimum solution for the main function $\Bar{H}_{\lambda}(\, \cdot \,)$. If the optimal solution is not unique, this regularization is necessary to guarantee that the maximum $\lambda$-cent-dian always belongs to $\mathbfcal{PO}^{\alpha}_2$. Besides, each subgraph $\mathcal{S}\in\mathbfcal{PO}^{\alpha}_2$ can be found as a maximum $\lambda$-cent-dian with $0\leq\lambda\leq1$. The proofs of these last two statements are detailed below in Propositions \ref{prop:PO_maximum} and \ref{prop:maximum_PO}. By construction of the functions used, they are similar to the ones exposed for Propositions 3 and 4 in \cite{ogryczak1997centdians} for Location Theory.

\begin{proposition}\label{prop:PO_maximum}
	For each $\lambda\in(0,1)$, the corresponding maximum $\lambda$-cent-dian belongs to $\mathbfcal{PO}^{\alpha}_2$.
\end{proposition}
\begin{proof}
	Let $\mathcal{S}\in\mathbfcal{N}^{\alpha}$ be a maximum $\lambda$-cent-dian for some $\lambda\in(0,1)$. Let us prove by \textit{Reductio ad absurdum}. That is, suppose that $\mathcal{S}\notin\mathbfcal{PO}^{\alpha}_2$. This means that there exists $\mathcal{S}'\in\mathbfcal{N}^{\alpha}$ such that
	$$F_c(\mathcal{S}')\leq F_c(\mathcal{S})\quad\text{ and }\quad F_m(\mathcal{S}')\leq F_m(\mathcal{S}),$$
	where at least one of the inequalities is satisfied strictly. Hence, since $\lambda\in(0,1)$,
	$$\Bar{H}_{\lambda}(\mathcal{S}')\leq \Bar{H}_{\lambda}(\mathcal{S})\quad\text{ and }\quad H_{\lambda}(\mathcal{S}')\leq H_{\lambda}(\mathcal{S}),$$
	which contradicts the fact that $\mathcal{S}$ is a maximum $\lambda$-cent-dian. Thus, a maximum $\lambda$-cent-dian solution network always belongs to $\mathbfcal{PO}^{\alpha}_2$.
\end{proof}

\begin{proposition}\label{prop:maximum_PO}
	For each $\mathcal{S}\in\mathbfcal{PO}^{\alpha}_2$ there exists $\lambda\in(0,1)$ such that $\mathcal{S}$ is the corresponding maximum $\lambda$-cent-dian.
\end{proposition}
\begin{proof}
	Let us consider $\mathcal{S}\in\mathbfcal{PO}^{\alpha}_2$ and $\lambda=F_m(\mathcal{S})/(F_c(\mathcal{S})+F_m(\mathcal{S}))$. Observe that $\lambda\in(0,1)$ and $1-\lambda=F_c(\mathcal{S})/(F_c(\mathcal{S})+F_m(\mathcal{S}))$. Then,
	\begin{equation}\label{eq:prop_maximum_PO}
		\Bar{H}_{\lambda}(\mathcal{S})=F_c(\mathcal{S})F_m(\mathcal{S})/(F_c(\mathcal{S})+F_m(\mathcal{S}))=\lambda F_c(\mathcal{S})=(1-\lambda)F_m(\mathcal{S}).
	\end{equation}
	Let us prove it by \textit{Reductio ad absurdum}. Let us suppose that $\mathcal{S}$ is not the corresponding maximum $\lambda$-cent-dian. This means that there exists $\mathcal{S}'\in \mathbfcal{N}^{\alpha}$ such that
	$$\Bar{H}_{\lambda}(\mathcal{S}')\leq \Bar{H}_{\lambda}(\mathcal{S})\quad\text{ and }\quad H_{\lambda}(\mathcal{S}')\leq H_{\lambda}(\mathcal{S}),$$
	where at least one of the inequalities is satisfied strictly. That is,
	\begin{equation}
		\lambda F_c(\mathcal{S}')\leq \Bar{H}_{\lambda}(\mathcal{S})\quad\text{ and }\quad (1-\lambda)F_m(\mathcal{S}')\leq H_{\lambda}(\mathcal{S}),
	\end{equation}
	where at least one of the inequalities has to be satisfied strictly. By equation \eqref{eq:prop_maximum_PO}, it would mean that $\mathcal{S}\notin\mathbfcal{PO}^{\alpha}_2$. Thus, $\mathcal{S}$ has to be the corresponding maximum $\lambda$-cent-dian.
\end{proof}

Regarding Example $5$, we have shown that the subnetwork $\mathcal{S}_1$ cannot be a $\lambda$-cent-dian for any $0\leq\lambda\leq1$ since $H_{\lambda}(\mathcal{S}_m)< H_{\lambda}(\mathcal{S}_1)$ or $H_{\lambda}(\mathcal{S}_c)< H_{\lambda}(\mathcal{S}_1)$. Note that $\Bar{H}_{\lambda}(\mathcal{S}_1)=\max\{80\,\lambda,56.52\,(1-\lambda)\}$, $\Bar{H}_{\lambda}(\mathcal{S}_c)=\max\{70\,\lambda,65.21\,(1-\lambda)\}$ and $\Bar{H}_{\lambda}(\mathcal{S}_m)=\max\{92\,\lambda,21.39\,(1-\lambda)\}$. Hence, $\Bar{H}_{\lambda}(\mathcal{S}_1)< \Bar{H}_{\lambda}(\mathcal{S}_c)$ and $\Bar{H}_{\lambda}(\mathcal{S}_1)< \Bar{H}_{\lambda}(\mathcal{S}_m)$ for any $0.4467<\lambda<0.4491$. In fact, subnetwork $\mathcal{S}_1$ is the maximum $\lambda$-cent-dian, for $\lambda \in (0.4467, \,0.4491)$.

As a conclusion, similar to the $\lambda$-cent-dian concept, the maximum $\lambda$-cent-dian generates the solution network depending on the value of $\lambda\in(0,1)$. Nevertheless, the difference between both concepts is that the maximum $\lambda$-cent-dian allows us to model all the existing compromises between $F_c(\,\cdot\,)$ and $F_m(\,\cdot\,).$ 

In the following example, we identify the $\mathbfcal{PO}_2^{\alpha}$ set for a given network.

\begin{examplecont}{5 cont}
Regarding the instance of Example 5 but considering $\alpha C_{total}=90$, the set of 17 points in Figure \ref{fig:Pareto_optimality} represents the whole set $\mathcal{N}^{\alpha}$. That is, solutions whose construction cost is less than or equal to $\alpha\,C_{total}$. We have highlighted the only three solutions that are not dominated. They form the $\mathcal{PO}_2^{\alpha}$ set. Their construction cost is larger than or equal to the ones that are dominated. They are the corresponding maximum $\lambda$-cent-dians for $\lambda\in[0.0001,0.3805]$, $\lambda\in[0.3806,0.4354]$ and $\lambda\in[0.4355,0.9999]$. 

Observing the whole set of feasible solutions, we identify that solutions labeled 2 and 7 are both center solutions. The difference is that in solution 2 we require a small percentage of efficiency with constraint \eqref{eq:delta_constraint}. As this percentage grows, we will consequently obtain solution points 3, 6, 10 and 4. With respect to median solutions, point 4 is the most efficient one, it represents the median solution. 
Regarding the generalized-center solution, point 7 corresponds to the solution that has the smallest difference between the values of the center and the median. In this case, the generalized-center corresponds with the less efficient center. We can require also a percentage of efficiency for the generalized-center solution. As this percentage grows, we will consequently obtain solutions points 2, 3, 6, 10 and 4. The rest of the points are simply feasible solution networks whose cost is strictly less than the available budget. So, they do not correspond to the optimal solution for any of the three concepts studied.
\end{examplecont}

\begin{figure}[ht]
	\centering
	\hspace*{-0.3cm}\includegraphics[scale=0.4]{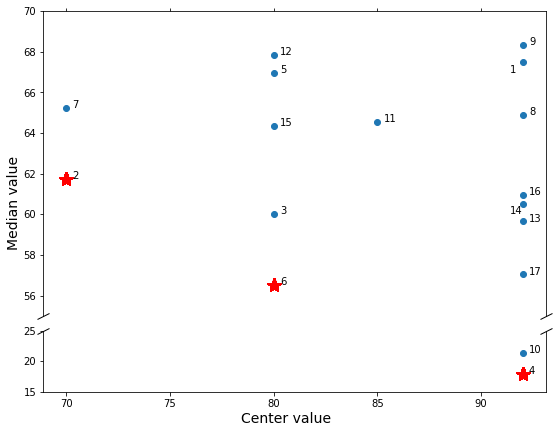}
	\caption{Identification of Pareto-optimality solutions for the given network of Example 5.}
	\label{fig:Pareto_optimality}
\end{figure}

\subsection{Adding efficiency to the generalized-center problem.}\label{subsec:efficiency}

In our prior analysis, we observed peculiarities in the solutions of the generalized-center concerning to the center and median functions. When the search domain for the generalized-center is confined to the set $\mathcal{PO}^{\alpha}_2$, it aligns with the least efficient center network. We study the effect of imposing constraints that limit the degree of inefficiency of the sub-network, by adding the following constraint: 
\begin{equation}\label{eq:delta_constraint}	F_m(\mathcal{S}^*)\leq(1+\Delta)\,F_m(\mathcal{S}_m), \quad\text{ with } \Delta\geq 0. 
\end{equation}

By setting $\Delta=0$, we ensure that the median objective value in $\mathcal{S}^*$ is precisely equivalent to that in $\mathcal{S}_m$. Note that, by definition, $F_m(\mathcal{S}_m)\leq F_m(\mathcal{S}), \forall,\mathcal{S}\in\mathbfcal{N}^{\alpha}$. When $\Delta\in(0,1)$, we concede a certain percentage of efficiency. In other words, the median objective value in the optimal solution network can exceed the median network value by a specified percentage. Lastly, if $\Delta\geq1$, we permit the efficiency value to surpass that of the median network by a multiple.

\section{Complexity discussion}\label{sec:complexity}

The problems addressed in this paper are NP-hard. The objective functions of the $\lambda$-cent-dians and generalized center problems are linear combinations of the median and center functions. Thus, if both network design problems are NP-hard, then so are the $\lambda$-cent-dians and generalized center problems. In Propositions  \ref{prop:medianNPhardV2} and \ref{prop:centerNPhardV2}  we show that the p-median and p-center problems are special cases of median and center subnetwork problems. It is known that the p-median and p-center problems are NP-hard (see \cite{kariv1979algorithmic_center} and \cite{kariv1979algorithmic_median}).

\begin{proposition}\label{prop:medianNPhardV2}
	The problem of finding a median subnetwork from a given network is NP-hard.
\end{proposition} 

\begin{proof}
	We show that the p-Median Problem reduces to the problem of finding a median subnetwork. 
	
	Given a set $I$ of demand nodes, let denote by $d_{ij} \geq 0$ the distance between demand nodes $i, j \in I$, with $d_{ij}=0$, if $i=j$. The decision version of the p-Median problem consists in determining a subset $S \subseteq I$, such that (i) $ |S|=p $, and (ii) $f_m(S) = \sum_{i \in I}\min_{j\in S}d_{ij} \leq K$.
	
	The corresponding median subnetwork problem is obtained as follows. 
	
	Regarding the network, the node set $N = I\cup \{t\}$ and the edge set $E = \{(i, j): i,j  \in I\cup \{t\}\}$. The costs $c_{it}= 1$ for all nodes $i \in I$ and the cost of all the other edges (linking pairs of nodes $i,j \in I$) are equal to zero.  The  cost $b_i =0$ for all nodes $i \in I\cup \{t\}$. The length of the edges linking pair of nodes i and j from $I$ is given by $d_{ij}$ and we set $d_{it}=0$ for all $i \in I$. 
	
	The set of O/D pairs $W = \{ (i, t): i \in I\}$ and each pair $w$ has a unit demand, $g^w = 1$ and an arbitrary large private utility, $u^w >> 1$.
	It is easy to see that there exists a subset  $S$ satisfying the above conditions (i) and (ii) if and only if there exists a  subnetwork $S'$ with a cost lower than or equal to $\alpha\, C_{total} = p$ and $F_m(S') \leq K$. In fact, each node that is selected as a median incurs in a cost equal to 1, and all the nodes that are connected to these median points are connected at cost zero. The network $S'$ contains all nodes in $I\cup \{t\}$, all edges between pairs of nodes in $I$ and $p$ edges linking the nodes in $S$ to the destination node $t$. Further, $f_m(S) = F_m(S')$.
\end{proof}

\begin{proposition}\label{prop:centerNPhardV2}
	The problem of finding a center subnetwork from a given network is NP-hard.
\end{proposition} 

\begin{proof}
	The p-center problem reduces to the center subnetwork problem. The reduction is exactly the same as that of the median version given in the previous proposition.
\end{proof}

Propositions 6 and 7 are evidence that solving this family of problems is hard. In \cite{bucarey2024on2}, we will investigate numerical algorithms to solve the problems exposed in this article. 

\section{Applications}\label{sec:applications}

In this section, we present a real-world application based on the concepts discussed earlier. The example involves a social planner who aims to design a metro network accessible to potential riders via pedestrian pathways. Riders typically travel to their nearest station, but they are unwilling to walk beyond a certain distance. To address this, a pedestrian walking threshold, denoted by $k \geq 0$, is established to limit the maximum walking distance to access the stations.

For each subnetwork $\mathcal{S}$ and pair $w\in W,$ $w\notin N_{\mathcal{S}}\times N_{\mathcal{S}},$ let $i^s_{\mathcal{S}}=\argmin_{i\in \mathcal{S}} d_{\mathcal{N}}(i,w^s), \mbox{ and }i^t_{\mathcal{S}}=\argmin_{i\in \mathcal{S}} d_{\mathcal{N}}(i,w^t),$ be the closest nodes in $\mathcal{S}$ to $w^s$ and $w^t$, respectively. If $\max\{d_{\mathcal{N}}(i^s_{\mathcal{S}},\mathcal{S}), d_{\mathcal{N}}(i^t_{\mathcal{S}},\mathcal{S})\}\le k,$ then $d'_{\mathcal{S}}(w)$ is that given by $d_{\mathcal{N}}(w^s,i^s_{\mathcal{S}})$ plus the shortest path between $i^s_{\mathcal{S}}$ and $i^t_{\mathcal{S}}$ plus $d_{\mathcal{N}}(i^t_{\mathcal{S}},w^t)$.
Finally, the general definition of the extended distance through $\mathcal{S}$ follows:
$$ d'_{\mathcal{S}}(w)= \left\{  \begin{array}{ll} d_{\mathcal{S}}(w), 
 &\mbox{ if }  w^s, w^t\in  {\mathcal{S}},  \\
\beta\, d_{\mathcal{N}}(i^s_{\mathcal{S}},w^s)+ d_{\mathbfcal{S}}(i^s_{\mathcal{S}},w^t), &\mbox{ if } w^s\notin \mathcal{S}, w^t\in \mathcal{S} \mbox{ and } d_{\mathcal{N}}(i^s_{\mathcal{S}}, w^s)\le k, \\
\beta\,d_{\mathcal{N}}(i^t_{\mathcal{S}},w^t)+ d_{\mathcal{S}}(i^t_{\mathcal{S}},w^s), &\mbox{ if } w^t\notin \mathcal{S}, w^s\in \mathcal{S} \mbox{ and } d_{\mathcal{N}}(i^t_{\mathcal{S}}, w^t)\le k,\\ 
\begin{split}
    & \beta\, d_{\mathcal{N}}(i^s_{\mathcal{S}},w^s)+ d_{\mathcal{S}}(i^s_{\mathcal{S}},i^t_{\mathcal{S}})+ \\
    & +\beta\,d_{\mathcal{N}}(i^t_{\mathcal{S}},w^t),
\end{split}& \begin{split}
    & \mbox{ if }  
w^s, w^t\notin \mathcal{S}, \mbox{ and }\\  & \max\{d_{\mathcal{N}}(i^s_{\mathcal{S}}, w^s),d_{\mathcal{N}}(i^t_{\mathcal{S}},w^t) \}\le k,
\end{split} \\
+\infty, & \mbox{ otherwise.}
\end{array} \right.$$

Let us remark, that the parameter $k$ can be large enough to reach all demand pairs, and that by introducing a parameter $\beta > 1$ we can penalize the pedestrian branches of the routes between the elements of each pair.

We show an example in Figure \ref{fig:application}. The thick line represents the design of the metro network. Then, there are only four metro stations. The rest of the nodes represent the set of potential centroids. For simplicity, we consider that the set of O/D pairs is composed of $w_1=(1,3)$, $w_2=(6,3)$, $w_3=(6,7)$ and $w_4=(8,4)$. If we set $k=35$ and $\beta=2$, the dashed line represents the feasible paths to reach some origin centroids in the metro network by the pedestrian mode. The numeric quantity in the nodes refers to their labels, and those in the edges are distances. It is known that $u^{w_1}=70$, $u^{w_2}=70$, $u^{w_3}=80$ and $u^{w_4}=80$. By evaluating $d'_{\mathcal{S}}(\cdot)$ for them we obtain that:
\begin{itemize}
    \item $d'_{\mathcal{S}}(w_1)= +\infty$ because $\beta\, d_{\mathcal{N}}(1,2)>k$. Then, it cannot access the metro network, $\ell_{\mathcal{S}}(w_1)= \min\{+\infty,70\}=70$ and it is not covered.
    \item $d'_{\mathcal{S}}(w_2)=\beta\, d_{\mathcal{N}}(6,5)+d_{\mathcal{S}}(5,4)+d_{\mathcal{S}}(4,3)=75$. You can check that there is another feasible path, but longer than the chosen one. Then, $\ell_{\mathcal{S}}(w_2)= \min\{75,70\}=70$ and it is not covered.
    \item $d'_{\mathcal{S}}(w_3)= \beta\, d_{\mathcal{N}}(6,5)+d_{\mathcal{S}}(5,4)+\beta\,d_{\mathcal{N}} (4,7)=76$. There are no more feasible and competitive paths for this O/D pair. Then, $\ell_{\mathcal{S}}(w_3)= \min\{76,80\}=76$ and it is covered.
    \item $d'_{\mathcal{S}}(w_4)=\beta\, d_{\mathcal{N}}(8,3)+d_{\mathcal{S}}(3,4)=57$. There are no more feasible edges to access the metro network because the only existing edge that connects the centroid to the metro network, apart from $(8,3)$ is $(8,2)$ but $\beta\,d_{(8,2)}>k$. Then, $\ell_{\mathcal{S}}(w_4)= \min\{57,80\}=57$ and it is covered.
\end{itemize}
In this situation, finally, pairs $w_3$ and $w_4$ are covered by the metro network by combining the pedestrian and the metro network modes. That is, they can access the metro network by walking, and the total distance from each origin to each destination is lower than its associated utility.

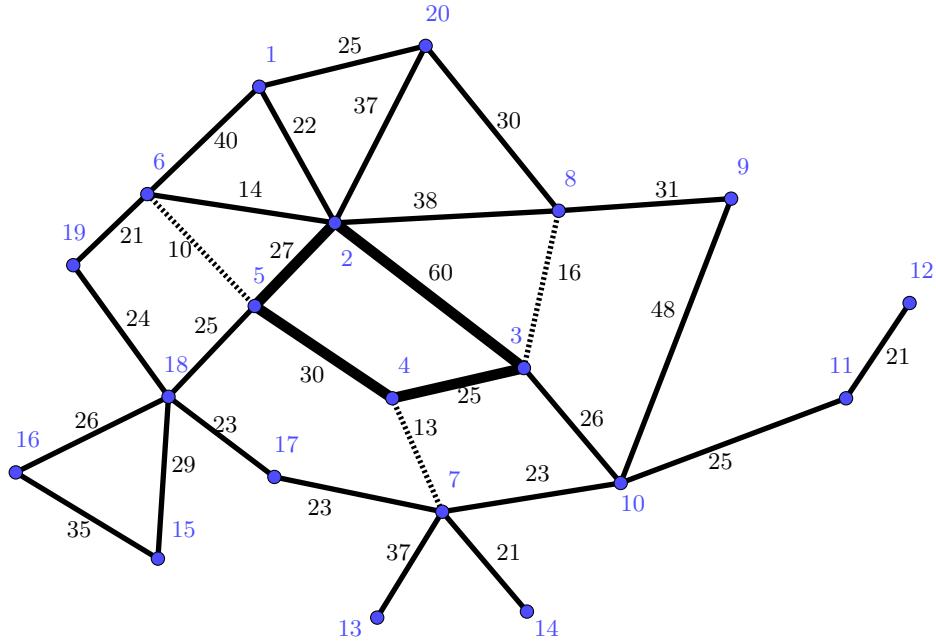
\begin{figure}[ht!]
    \definecolor{ududff}{rgb}{0.30196078431372547,0.30196078431372547,1}
\begin{tikzpicture}
\draw [line width=4pt] (-6.3,2.36)-- (-7.36,1.26);
\draw [line width=4pt] (-7.36,1.26)-- (-5.54,0.04);
\draw [line width=4pt] (-5.54,0.04)-- (-3.8,0.44);
\draw [line width=4pt] (-3.8,0.44)-- (-6.3,2.36);
\draw [line width=2pt] (-6.3,2.36)-- (-7.3,4.16);
\draw [line width=2pt] (-7.3,4.16)-- (-5.1,4.7);
\draw [line width=2pt] (-5.1,4.7)-- (-6.3,2.36);
\draw [line width=2pt] (-6.3,2.36)-- (-8.78,2.74);
\draw [line width=2pt,dash pattern=on 1pt off 1pt] (-8.78,2.74)-- (-7.36,1.26);
\draw [line width=2pt] (-7.36,1.26)-- (-8.5,0.06);
\draw [line width=2pt] (-8.5,0.06)-- (-9.76,1.8);
\draw [line width=2pt] (-9.76,1.8)-- (-8.78,2.74);
\draw [line width=2pt] (-8.78,2.74)-- (-7.3,4.16);
\draw [line width=2pt] (-8.5,0.06)-- (-10.52,-0.94);
\draw [line width=2pt] (-10.52,-0.94)-- (-8.64,-2.08);
\draw [line width=2pt] (-8.64,-2.08)-- (-8.5,0.06);
\draw [line width=2pt] (-8.5,0.06)-- (-7.1,-1);
\draw [line width=2pt] (-7.1,-1)-- (-4.88,-1.46);
\draw [line width=2pt,dash pattern=on 1pt off 1pt] (-4.88,-1.46)-- (-5.54,0.04);
\draw [line width=2pt] (-3.8,0.44)-- (-2.52,-1.08);
\draw [line width=2pt] (-2.52,-1.08)-- (-4.88,-1.46);
\draw [line width=2pt] (-4.88,-1.46)-- (-5.74,-2.86);
\draw [line width=2pt] (-4.88,-1.46)-- (-3.76,-2.78);
\draw [line width=2pt,dash pattern=on 1pt off 1pt] (-3.8,0.44)-- (-3.34,2.52);
\draw [line width=2pt] (-3.34,2.52)-- (-6.3,2.36);
\draw [line width=2pt] (-5.1,4.7)-- (-3.34,2.52);
\draw [line width=2pt] (-3.34,2.52)-- (-1.06,2.68);
\draw [line width=2pt] (-1.06,2.68)-- (-2.52,-1.08);
\draw [line width=2pt] (-2.52,-1.08)-- (0.46,0.04);
\draw [line width=2pt] (0.46,0.04)-- (1.3,1.3);
\begin{scriptsize}
\draw [fill=ududff] (-6.3,2.36) circle (2.5pt);
\draw[color=ududff] (-6.14,1.89) node {$2$};
\draw [fill=ududff] (-7.36,1.26) circle (2.5pt);
\draw[color=ududff] (-7.3,1.66) node {$5$};
\draw[color=black] (-7,2.) node {$27$};
\draw [fill=ududff] (-5.54,0.04) circle (2.5pt);
\draw[color=ududff] (-5.38,0.47) node {$4$};
\draw[color=black] (-6.6,0.35) node {$30$};
\draw [fill=ududff] (-3.8,0.44) circle (2.5pt);
\draw[color=ududff] (-3.9,0.87) node {$3$};
\draw[color=black] (-4.52,0.08) node {$25$};
\draw[color=black] (-4.9,1.7) node {$60$};
\draw [fill=ududff] (-7.3,4.16) circle (2.5pt);
\draw[color=ududff] (-7.14,4.59) node {$1$};
\draw[color=black] (-6.7,3.65) node {$22$};
\draw [fill=ududff] (-5.1,4.7) circle (2.5pt);
\draw[color=ududff] (-4.94,5.13) node {$20$};
\draw[color=black] (-6.1,4.7) node {$25$};
\draw[color=black] (-5.89,3.91) node {$37$};
\draw [fill=ududff] (-8.78,2.74) circle (2.5pt);
\draw[color=ududff] (-8.62,3.17) node {$6$};
\draw[color=black] (-7.42,2.8) node {$14$};
\draw[color=black] (-8.35,2.01) node {$10$};
\draw [fill=ududff] (-8.5,0.06) circle (2.5pt);
\draw[color=ududff] (-8.4,0.49) node {$18$};
\draw[color=black] (-8,1.) node {$25$};
\draw [fill=ududff] (-9.76,1.8) circle (2.5pt);
\draw[color=ududff] (-9.75,2.23) node {$19$};
\draw[color=black] (-8.9,1.1) node {$24$};
\draw[color=black] (-8.98,2.2) node {$21$};
\draw[color=black] (-7.74,3.45) node {$40$};
\draw [fill=ududff] (-10.52,-0.94) circle (2.5pt);
\draw[color=ududff] (-10.36,-0.51) node {$16$};
\draw[color=black] (-9.58,-0.25) node {$26$};
\draw [fill=ududff] (-8.64,-2.08) circle (2.5pt);
\draw[color=ududff] (-8.3,-1.7) node {$15$};
\draw[color=black] (-9.68,-1.7) node {$35$};
\draw[color=black] (-8.3,-0.83) node {$29$};
\draw [fill=ududff] (-7.1,-1) circle (2.5pt);
\draw[color=ududff] (-6.94,-0.57) node {$17$};
\draw[color=black] (-7.75,-0.3) node {$23$};
\draw [fill=ududff] (-4.88,-1.46) circle (2.5pt);
\draw[color=ududff] (-4.72,-1.03) node {$7$};
\draw[color=black] (-6.5,-1.4) node {$23$};
\draw[color=black] (-5.1,-0.35) node {$13$};
\draw [fill=ududff] (-2.52,-1.08) circle (2.5pt);
\draw[color=ududff] (-2.36,-1.35) node {$10$};
\draw[color=black] (-2.9,-0.23) node {$26$};
\draw[color=black] (-3.62,-0.95) node {$23$};
\draw [fill=ududff] (-5.74,-2.86) circle (2.5pt);
\draw[color=ududff] (-6.1,-3) node {$13$};
\draw[color=black] (-5.46,-2) node {$37$};
\draw [fill=ududff] (-3.76,-2.78) circle (2.5pt);
\draw[color=ududff] (-3.5,-3) node {$14$};
\draw[color=black] (-4,-2) node {$21$};
\draw [fill=ududff] (-3.34,2.52) circle (2.5pt);
\draw[color=ududff] (-3.18,2.95) node {$8$};
\draw[color=black] (-3.2,1.71) node {$16$};
\draw[color=black] (-5.1,2.65) node {$38$};
\draw[color=black] (-4,3.71) node {$30$};
\draw [fill=ududff] (-1.06,2.68) circle (2.5pt);
\draw[color=ududff] (-0.9,3.11) node {$9$};
\draw[color=black] (-1.9,2.8) node {$31$};
\draw[color=black] (-1.96,1.21) node {$48$};
\draw [fill=ududff] (0.46,0.04) circle (2.5pt);
\draw[color=ududff] (0.4,0.47) node {$11$};
\draw[color=black] (-1.2,-0.8) node {$25$};
\draw [fill=ududff] (1.3,1.3) circle (2.5pt);
\draw[color=ududff] (1.46,1.73) node {$12$};
\draw[color=black] (1.15,0.6) node {$21$};
\end{scriptsize}
\end{tikzpicture}
    \caption{Example of accessibility to a metro network.}
    \label{fig:application}
\end{figure}

\section{Conclusions}\label{sec:conclusions} 

In this paper, we introduced and studied the $\lambda$-cent-dian and the generalized center problems in network design for the first time. Both problems aim to minimize a linear combination of the maximum and average traveled distances. The $\lambda$ cent-dian problem, where $\lambda\in [0,1]$, minimizes a convex combination of these objectives, while the generalized center problem minimizes the difference between them. We explored these concepts under two versions of Pareto-optimality: the first version considers the shortest paths of each origin/destination (O/D) pair, while the second version addresses both objective functions simultaneously. Regarding the second version, we found that the introduction of the new concept of maximum $\lambda$ cent-dian is necessary to generate the entire set of Pareto-optimal solutions. We include an example that shows the applicability of the mathematical models we have studied in this paper, thus completing the twofold contribution of this paper.

\section{Compliance with Ethical Standards}

This work is partially supported by Ministerio de Ciencia e Innovaci\'on under grant PID2020-114594GB-C21 funded by MICIU/AEI/10.13039/501100011033, grant US-1381656 funded by Programa Operativo FEDER/Andalucía and grant ANID PIA AFB230002 funded by the Instituto Sistemas Complejos de Ingenier\'ia.

All authors declare that they have no conflict of interest.

This article does not contain any studies with human participants or animals performed by any of the authors.

\bibliography{sn-bibliography}

\end{document}